\documentclass[11pt]{amsart}

\usepackage[utf8]{inputenc} 			
\usepackage[T1]{fontenc} 			
\usepackage[francais, english]{babel} 			
\usepackage{layout}					
\usepackage{setspace}				
\usepackage{graphics}				
\usepackage{cite}

\usepackage[hyperindex,breaklinks]{hyperref}

\usepackage{amsthm}
\usepackage{amsmath}
\usepackage{amssymb}
\usepackage{mathrsfs}
\usepackage{dsfont}			
\usepackage{enumerate}

\usepackage{wasysym}

\usepackage{mathtools,amscd}            
\usepackage{tikz-cd}            

\usepackage{hyperref}		
\hypersetup{
    colorlinks,
    citecolor=blue,
    filecolor=blue,
    linkcolor=blue,
    urlcolor=blue
}

\usepackage{geometry}				

\newcommand\C{\mathbb{C}}
\newcommand\Q{\mathbb{Q}}

\newcommand\N{\mathbb{N}}

\newcommand\OO{\mathcal{O}}

\newcommand\G{\mathbb{G}}

\newcommand\LL{\mathscr{L}}

\newcommand\Z{\mathbb{Z}}

\newcommand{\set}[1]{\left\{ {#1} \right\}}
\newcommand{\vect}[1]{\langle {#1} \rangle}
\newcommand{\abs}[1]{\lvert {#1} \rvert}

\def\Ind#1#2{#1\setbox0=\hbox{$#1x$}\kern\wd0\hbox to 0ex{\hss$#1\mid$\hss}
\lower.9\ht0\hbox to 0ex{\hss$#1\smile$\hss}\kern\wd0}
\def\Notind#1#2{#1\setbox0=\hbox{$#1x$}\kern\wd0\hbox to 0ex{\mathchardef\nn="0236\hss$#1\nn$\kern1.4\wd0\hss}\hbox to 0ex{\hss$#1\mid$\hss}\lower.9\ht0
\hbox to 0ex{\hss$#1\smile$\hss}\kern\wd0}

\def\indi#1{\mathop{\ \ \hbox to 0ex{\hss$\vert^{\hbox to 0ex{$\scriptstyle#1$\hss}}$\hss}
\lower1ex\hbox to 0ex{\hss$\smile$\hss}\ \ }}

\def\nindi#1{\mathop{\ \ \hbox to 0ex{\hss$\!\not{\vert}^{\hbox to 0ex{$\scriptstyle\,#1$\hss}}$\hss}
\lower1ex\hbox to 0ex{\hss$\smile$\hss}\ \ }}

\newcommand{\findep}[1][]{%
  \mathrel{
    \mathop{
      \vcenter{
        \hbox{\oalign{\noalign{\kern-.3ex}\hfil$\vert$\hfil\cr
              \noalign{\kern-.7ex}
              $\smile$\cr\noalign{\kern-.3ex}}}
      }
    }\displaylimits_{#1}
  }
}

\newcommand{\nfindep}[1][]{%
  \mathrel{
    \mathop{
      \vcenter{
	\hbox{\oalign{\noalign{\kern-.3ex}\hfil$\!\not{\vert}$\hfil\cr
              \noalign{\kern-.7ex}
              $\smile$\cr\noalign{\kern-.3ex}}}
      }
    }\displaylimits_{#1}
  }
}

\makeatletter
\newcommand{\setword}[2]{%
  \phantomsection
  #1\def\@currentlabel{\unexpanded{#1}}\label{#2}%
}
\makeatother

\newcommand\LLr{\mathscr{L}_{\mathrm{ring}}}

\newcommand{\NSOP}[1]{\mathrm{NSOP}_{#1}}  

\newcommand{\ACFG}{\mathrm{ACFG}}
\newcommand{\ACF}{\mathrm{ACF}}

\newcommand{\acl}{\mathrm{acl}}

\newcommand{\End}{\mathrm{End}}

\newcommand{\Id}{\mathrm{Id}}

\newcommand{\Frob}{\mathrm{Frob}}
\newcommand{\Mat}{\mathrm{Mat}}

\newtheorem{df}{Definition}[section]
\newtheorem{prop}[df]{Proposition}		

\newtheorem{lm}[df]{Lemma} 
\newtheorem{cor}[df]{Corollary} 
\newtheorem{fact}{Fact} 
\newtheorem{axioms}[df]{Axioms} 
 
\theoremstyle{remark}
\newtheorem{rk}[df]{Remark}
\theoremstyle{remark}

\newsavebox{\auteurbm}

\let\oldabstract\abstract
\let\oldendabstract\endabstract
\makeatletter
\renewenvironment{abstract}
{%
               {\list{}{\addtolength{\leftmargin}{8em} 
                        \listparindent 1.5em%
                        \itemindent    \listparindent%
                        \rightmargin   \leftmargin%
                        \parsep        \z@ \@plus\p@}%
                \item\relax}%
               {\endlist}%
\oldabstract}
{\oldendabstract}
\makeatother

\title{Generic expansion of an abelian variety by a subgroup}

\author{Christian d'Elb\'ee}\thanks{Einstein Institut, Hebrew University of Jerusalem.}

\date{\small\today}

\begin{document}

\maketitle				

\begin{abstract}
  Let $A$ be an abelian variety in a field of characteristic $0$. We prove that the expansion of $A$ by a generic divisible subgroup of $A$ with the same torsion exists provided $A$ has few algebraic endomorphisms, namely $\End(A)=\Z$. The resulting theory is $\NSOP 1$ and not simple. Note that there exist abelian varieties $A$ with $\End(A) = \Z$ of any genus. We indicate how this result can be extended to any simple abelian variety by considering the expansion by a predicate for some submodule over $\End(A)$.
\end{abstract}

This short note has two aims. First, it contributes to the recent interest in new $\NSOP 1$ theories arising from generic constructions~\cite{CR16}, \cite{KR17}, \cite{KrR18}, \cite{CKr17} and it follows the author's work in~\cite{dE18A} \cite{dE18B}. Further, it witnesses an appreciated phenomenon: an algebraic constraint has a model-theoretic consequence. 

\section{Freeness in abelian varieties}

Let $K$ be an algebraically closed field of characteristic $0$ and let $A$ be an abelian variety defined over a subfield $k_0$ of $K$. Let $\oplus,\ominus,e$ be the operations and neutral element in $A(K)$, and $\LL$ be the language consisting of $\set{\oplus, \ominus, e}$ and a predicate $W$ for each subvarieties of $A(K)^n$ defined over $k_0$, for each $n\in \N$. We consider the natural $\LL$-structure $(A(K),\oplus, \ominus, e, (W_t(K))_t)$ on $A(K)$, and let $T$ be the $\LL$-theory of this structure. $T$ has quantifier elimination in $\LL$ and does not depend on the choice of $K\supset k_0$. See~\cite[2.1.2]{Cay14} for a precise discussion on that setting.


An abelian variety is \emph{simple} if it is not isogenous to a product of two abelian varieties. This is equivalent to saying that its only abelian subvarieties are itself and $\set{e}$. Poincaré's reducibility theorem (see for instance~\cite[Chapter 5]{B98}) implies that every abelian variety is isogenous to a finite product of simple abelian varieties. 

The algebraic group structure on $A$ induces an algebraic group structure on each cartesian power $A^n$, and we denote again by $\oplus,\ominus,e$ the corresponding operations and constant.

\begin{df}
  Let $V$ be an irreducible subvariety of $A^n$. We say that $V$ is \emph{free} if it is not contained in a translate of an algebraic subgroup of $A^n$. 
\end{df}

Let $\End(A)$ be the ring of algebraic (or equivalently $\LL$-definable) endomorphismes of $A$. This ring only depends on $T$.  
If $\End(A) = \Z$ then $A$ is simple. The converse does not hold since elliptic curves are simple (see Remark~\ref{rk_ex}). In a simple abelian variety, every endomorphism is surjective and has finite kernel. The following is \cite[Lemma 4.1.(i)]{BHP18}.

\begin{fact}\label{fact_subgroups}  
Let $H$ be a connected algebraic subgroup of $A^n$. Then there exists $\theta\in \Mat_{n}(\End(A))$ such that $H = (\ker (\theta))^0$.
\end{fact}

\begin{prop}\label{prop_free}
  Assume that $A$ is simple. Let $V$ be an irreducible subvariety of $A^n$. Then the following are equivalent:
  \begin{enumerate}
    \item $V$ is free;
    \item for all generic $a$ of $V$ in $A'\succ_\LL A$ over $A$, for all $\sigma_1,\dots,\sigma_n\in \End(A)$, $\sigma_1 a_1\oplus \dots \oplus \sigma_n a_n\notin A$;
    \item there exists a generic $a$ of $V$ in $A'\succ_\LL A$ over $A$, such that for all $\sigma_1,\dots,\sigma_n\in \End(A)$, $\sigma_1 a_1\oplus \dots \oplus \sigma_n a_n\notin A$.
  \end{enumerate}
\end{prop}

\begin{proof}
  For \textit{(1)}$\implies$ \textit{(2)}, take a generic $a$ of $V$ in $A'\succ A$. If there exists $\sigma_1,\dots,\sigma_n\in \End(A)$ such that $\sigma_1 a_1\oplus \dots \oplus \sigma_n a_n =c \in A$, then $V$ is included in $H\oplus d$ where $H$ is the algebraic subgroup of $A^n$ defined by the equation $\sigma_1 x_1\oplus \dots \oplus \sigma_n x_n = e$ and $d = (d_1,e,\dots,e)\in A^n$ with $d_1\in \sigma_1^{-1}(c)$, contradicting freeness. \textit{(2)}$\implies$\textit{(3)} is clear as generics always exists. \textit{(3)}$\implies $\textit{(1)}. From \textit{(3)}, we get that $V$ is not included in the set defined by equations of the form $\sigma_1 x_1\oplus \dots \oplus \sigma_n x_n = c$ for $\sigma_1,\dots,\sigma_n\in \End(A)$ and $c\in A$. Assume that $V$ is not free, i.e. there exists an algebraic subgroup $H$ of $A^n$ such that $V\subset H \oplus d$ for some tuple $d\in A^n$. By irreducibility of $V$, we may assume that $H$ is connected. By Fact~\ref{fact_subgroups}, $H$ is included in a set defined by equations of the form $\sigma_1x_1\oplus\dots\oplus \sigma_n x_n = e$, a contradiction.
\end{proof}

\section{Model-companion for the case $\End(A) = \Z$}

Let $T^A$ be the theory of $(A,\oplus,\ominus,e)$. Let $\LL^G = \LL\cup\set{G}$, where $G$ is a unary predicate, and consider the expansion $T^G$ of $T$ to the language $\LL^G$ when $G$ is predicate for a subgroup of $(A,\oplus,\ominus,e)$, model of $T^A$. We prove that if $\End(A) = \Z$, then $T^G$ admits a model-companion. 

We start by some easy facts on the model theory of the structure $(A,\oplus, \ominus,e)$. Let $[n] : A \rightarrow A$ be the endomorphism defined by $[n]a = a\oplus \dots \oplus a$ ($n$ times).

\begin{prop}\label{prop_abelianMT}
  The group $(A,\oplus,\ominus,e)$ is divisible abelian and for each $[n]$, we have that $\ker [n] =: A[n]$ is finite and isomorphic to $(\Z/n\Z)^{2g}$. The theory $T^A$ of $(A,\oplus,\ominus,e)$ has quantifier elimination in the language $\set{\oplus, \ominus, e}$, hence it is strongly minimal, the algebraic closure defines a modular pregeometry defined by $$\vect{B} = \set{c\in A \mid [n_0]c \oplus [n_1]b_1 \oplus \dots \oplus [n_k]b_k = e,\mbox{ for some $b_i\in B$ and $n_0\neq 0,n_i\in \Z$}}$$
  for some $B\subset A$. If $A[\infty]:= \bigcup_{n\in \N} A[n]$ is the set of torsion points, then $A/A[\infty]$ has the structure of a $\Q$-vector space.
\end{prop}
\begin{proof}
  The first assertion is standard (see for instance \cite[Chapter 5]{B98}), and provides an axiomatisation of $T^A$. Quantifier elimination for $T^A$ is an easy exercise, and the rest of the proposition follows.
\end{proof}
\begin{rk}
  If the dimension of $A$ is one, then $T$ is also strongly minimal, hence there are two ambient pregeometries\footnote{In that case, the fact that $T^A$ is pregeometric follows also from the fact that a reduct of a pregeometric theory is always pregeometric~\cite[Fait 2.15]{Hi08}.}, the first one is the one in the sense of $T$ and the second one is in the sens of $T^A$. If the pregeometry in $T^A$ is modular, the one of $T$ never is. In this section when we consider independent tuples in the sense of the pregeomety in $A$, it will always mean in the sense of the modular one, i.e. in the pregeometry of the reduct $T^A$ of $T$. In this note, a tuple $a$ is independent over some subset $B$ if it is of maximal dimension over $B$ (in the sense of the pregeometry in models of $T^A$).
\end{rk}

The following lemma is not used in the proof of Proposition~\ref{prop_mc}, but it is worth mentionning.
\begin{lm}\label{lm_defsg}
  The following are equivalent.
  \begin{enumerate}
    \item $\End(A) = \Z$;
    \item every connected algebraic subgroup $H$ of $A^n$ is the connected component of a group defined by a formula of the form $\bigwedge_i [m_{i,1}] x_1 \oplus \dots \oplus [m_{i,n}]x_n = e$, for some $(m_{i,j})\in \Mat_{n}(\Z)$.
    \item every algebraic subgroup $H$ of $A^n$ is definable in the structure $(A,\oplus, \ominus, e)$;
  \end{enumerate}
\end{lm}
\begin{proof}
  \textit{(1)}$\implies$\textit{(2)}. Assume that $\End(A) = \Z$, and let $H$ be a connected algebraic subgroup of $A^n$. By Fact~\ref{fact_subgroups}, $H = (\ker(\theta))^0$.  By assumption, there exists $(m_{i,j})\in \Mat_{n,n}(\Z)$ such that $\ker(\theta) $ is the set of solutions of the formula $\bigwedge_i [m_{i,1}] x_1 \oplus \dots \oplus [m_{i,n}]x_n = e$, hence it is of the desired form. \textit{(2)}$\implies $\textit{(3)} Let $H$ be an algebraic subgroup of $A^n$. Then the connected component $H^0$ is of finite index in $H$, so it is sufficient to prove that $H^0$ is definable in $(A,\oplus, \ominus, e)$. From $\textit{(2)}$, $H^0$ is the connected component of the group $G$ defined by the formula $\bigwedge_i [m_{i,1}] x_1 \oplus \dots \oplus [m_{i,n}]x_n = e$. Now $H^0$ has finite index in $G$ hence there exists $k\in \N$ such that $H^0 = [k]G$, so $H^0$ is definable in the structure $(A,\oplus, \ominus, e)$. \textit{(3)}$\implies$\textit{(1)}. Assume that $\End(A)\neq \Z$, then from~\cite[Chapter 5]{B98}, as a $\Z$-module, $\End(A)$ is isomorphic to $\Z^r$ with $r>1$, hence there exists $\theta\in \End(A)$ such that $\theta$ and $\Id$ are $\Z$-linearly independent. Now the graph of $\theta$ is an algebraic subgroup of $A\times A$, assume that it is definable in the structure $(A,\oplus,\ominus,e)$. Let $a$ be a generic of $A$, then $\theta a$ is algebraic over $a$ in the sense of $(A,\oplus,\ominus,e)$. From Proposition~\ref{prop_abelianMT}, there exists $n,m\in \Z$ such that $[n]\theta a = [m]a$, hence $([n]\theta - [m]) a = e$. By genericity of $a$, we get $n \theta- m\Id = [0]$, a contradiction.
\end{proof}

The following is a mere translation of Proposition~\ref{prop_free}.
\begin{prop}\label{prop_free2}
  Assume that $A$ is an abelian variety with $\End(A) = \Z$. Let $V\subseteq A^n$ be an irreducible subvariety. Then the following are equivalent:
  \begin{enumerate}
    \item $V$ is free;
    \item for all generic $a$ of $V$ in $A'\succ_\LL A$, $a$ is independent over $A$ in the sense of the pregeometry of $(A,\oplus,\ominus,e)$;
    \item there exists a generic $a$ of $V$ in some $A'\succ_\LL A$, such that $a$ is independent over $A$ in the sense of the pregeometry of $(A,\oplus,\ominus,e)$.
  \end{enumerate}
\end{prop}

  An \emph{irreducible quasi-variety} is a subset of $A^n$ of the form $V\cap \OO$ where $V$ is an irreducible variety and $\OO$ is a Zariski open set. The Zariski closure of $V\cap \OO$ is $V$, and any generic of the variety $V$ is in $V\cap \OO$.

\begin{rk}\label{rk_quasi_def}
  Given any $\LLr$-formula $\phi(x,y)$ and field $K$, the set of tuples $b$ from $K$ such that the Zariski closure of $\phi(x,b)$ is irreducible, is definable~\cite[Theorem 10.2.1 (2)]{Joh16}. This implies that if a formula $\phi(x,b)$ defines an irreducible quasi-variety in a field $K$, there is a formula $\psi(x,y)$ such that $\phi(K,b) = \psi(K,b)$ and such that for any $b'$ in any field $K'\succ K$, $\psi(K',b')$ is an irreducible quasi-variety. Of course, this holds replacing $\LLr, K,K'$ by $\LL, A,A'$.
\end{rk}

\begin{df}
  We say that an $\LL$-formula $\psi(x,y)$ \emph{defines irreducible (quasi-)varieties in $x$} if for any $A\models T$ and tuple $b$ from $A$, the formula $\psi(x,b)$ defines an irreducible (quasi-)variety.
\end{df}

The following fact was observed in the proof of~\cite[Theorem 1.2]{BGH13}.

\begin{fact}\label{fact_free}
  Let $\phi(x,y)$ be an $\LL$-formula that defines irreducible varieties in $A^n$ in $x$. Then $\set{b \mid \phi(x,b) \text{ is free}}$ is $\LL$-definable.
\end{fact}

\begin{prop}\label{prop_H4}
  Assume that $\End(A) = \Z$. Let $\phi(x,y)$ be an $\LL$-formula that defines irreducible quasi-varieties in $x$. Then there exists an $\LL$-formula $\theta(y) = \theta_\phi(y)$ such that for all $B\models T$, for all $\abs{y}$-tuple $b$ from $B$, we have $B\models \theta(b)$ if and only if there exists $A'\succ_\LL B$ and a $\abs{x}$-tuple $a$ from $A'$ such that $a$ is independent over $B$, in the sense of the pregeometry of $(A',\oplus,\ominus,e)$.
\end{prop}

\begin{proof}
  Let $\phi(x,y)$ be as in the statement. From~\cite[Theorem 10.2.1 (1)]{Joh16}, there exists a formula $\psi(x,y)$ such that for all $b$, $\psi(x,b)$ is the Zariski closure of $\phi(x,b)$. Let $\theta(y)$ be the formula whose existence is stated in Fact~\ref{fact_free}. Let $B$ be a model of $T$ and $b$ a tuple from $B$ such that $B\models\theta(b)$. Then $\psi(x,b)$ defines a free irreducible variety. Let $a$ be a tuple from some elementary extension $A'$ of $B$, that is a generic of $\psi(x,b)$. By Proposition~\ref{prop_free2}, $a$ is independent over $B$ in the sense of the pregeometry $(A',\oplus,\ominus,e)$. Further, as $a$ is generic, such $a$ always satisfies $\phi(x,b)$. Conversely if $a$ is in some elementary extension $A'$ of $B$ such that $\phi(a,b)$ and $a$ is independent over $B$, then a generic of $\psi(x,b)$ is independent over $B$ hence $\psi(x,b)$ defines a free irreducible variety by Proposition~\ref{prop_free2}, hence $B\models \theta(b)$. 
\end{proof}

The result of Proposition~\ref{prop_H4} is, in general, the hard step in order to axiomatise existentially closed models of $T^G$, it corresponds to hypothesis $(H_4)$ in \cite{dE18A}. We gather now the data we have collected concerning the theory $T$ and its reduct $T^A$, when $\End(A) = \Z$:
\begin{enumerate}
  \item $T$ has quantifier elimination in $\LL$;
  \item $T^A$ has quantifier elimination in $\set{\oplus,\ominus,e}$, is strongly minimal, and the algebraic closure defines a modular pregeometry;
  \item Proposition~\ref{prop_H4}.
\end{enumerate}
  
Then, using~\cite[Theorem 1.5, Remark 1.8]{dE18A}, we get
\begin{prop} \label{prop_mc}
    Suppose that $\End(A) = \Z$. Then $T^G$ admits a model-companion, we call it $TG$.   
  \end{prop}
 \begin{rk}\label{rk_ex}
  Note that the condition $\End(A) = \Z$ can fail even for elliptic curves. For instance, consider the elliptic curve $E$ in $\C$ defined by the equation $y^2 = x^3 - x$. Then $[i] : (x,y) \mapsto (-x,iy)$ is an endomorphism of $E$, and we have $\End(E) = \Z[i]$ \cite[Example 4.4]{Sil09}. In positive characteristic, $\End(E)$ is always bigger than $\Z$ \cite[Theorem 3.1]{Sil09}. However, in characteristic $0$, there exists abelian varieties with $\End(A) = \Z$ of any genus, and those are the generic ones~\cite[Section 4.7]{Ara12}.
\end{rk}

We further provide a set of axioms for $TG$. Assume that $a,b$ are tuples of elements of $A$ such that $a$ is linearly independent over $b$. Then $u\in \vect{ab}\setminus \vect{b}$ if and only if there exists $n_0\neq 0$, $n_1,\dots,n_{\abs{a}}$ not all zero and $m_1,\dots,m_{\abs{b}}$ such that $$[n_0]u \oplus \bigoplus_{i=1}^{\abs{a}} [n_i]a_i \oplus \bigoplus_{j=1}^{\abs{b}} [m_j]b_j=e.$$
Let $\tau(t,x,y)$ be the formula $[n_0]t \oplus \bigoplus_{i=1}^{\abs{a}} [n_i]x_i \oplus \bigoplus_j [m_j]y_j=e$. We call\footnote{Following the terminology in~\cite{dE18A}.} the formula $\tau(t,x,y)$ an \emph{equation strict in $x$, (algebraic in $t$)}, which means that $n_0\neq 0$ and $n_1,\dots,n_{\abs{x}}$ are not all zero. We now give an axiomatisation of $TG$.

\begin{axioms}
  $TG$ is the theory consisting of $T^G$ together with the following axiom scheme: for each $\phi(x,y)$ such that $\phi(x,y)$ defines irreducible quasi-varieties in $x$, for all partition $x = x^0x^1$, for all $k$ and equations $\tau_1(t,x,y),\dots,\tau_k(t,x,y)$ strict in $x^1$:
$$\forall y (\theta_\phi(y) \rightarrow (\exists x \phi(x,y)\wedge x^0\subseteq G \wedge \bigwedge_i \forall t \ \tau_i(t,x,y)  \rightarrow t\notin G)).$$
\end{axioms}

For the sake of completeness, we give a proof of Proposition~\ref{prop_mc}.

\begin{proof}
  As in Theorem~\cite[Theorem 1.5]{dE18A} or other similar results, the proof is in two steps. 
    \begin{center} \emph{Step 1: $TG$ is consistent.}\end{center}
    Given one instance of the axiom-scheme, say for $\phi(x,y)$, and $x = x^0x^1$. If some model $(B,G)$ of $T^G$ satisfies $\theta(b)$ for some tuple $b$ from $B$. Then there exists $A'\succ_{\LL} B$ and a tuple $a$ from $A'$ such that $\phi(a,b)$ and $a$ is an independent tuple over $B$ in the sense of the pregeometry in $(A',\oplus,\ominus,e)$. The partition $a = a^0a^1$ is given and define $G' = \vect{G,a^0}\subset A'$. Then $(A',G')$ is a model of $T^G$ that satisfies the conclusion of the instance of the axiom for a given $b$ and any formula $\tau(t,x,y)$ strict in $x^1$. Taking union of chains, one shows that every model of $T^G$ embedds in a model of $TG$. 
    \begin{center}\emph{Step 2: every model of $TG$ is existentially closed in every extension model of $T^G$.}\end{center}
    Let $(B,G)\models TG$ and $(A',G')\models T^G$ extending $(B,G)$. Let $b$ be a tuple from $B$ and $u$ a tuple from $A'$. The quantifier-free type of $u$ over $b$ is given by formulae $\phi(x,b)$ of the form 
    $$\psi(x,b)\wedge \bigwedge_i P_i(x,b)\in G \wedge \bigwedge_j Q_j(x,b)\notin G$$
    for $\LL$-terms $P_i(x,y),Q_j(x,y)$, and $\psi(x,y)$ a quantifier-free $\LL$-formula. If $\theta(z,y)$ is the formula $\exists x \psi(x,y)\wedge \bigwedge_i z_i = P_i(x,y)\wedge \bigwedge_j z_j = Q_j(x,y)$, then $\exists x \phi(x,y)$ is equivalent to $$\exists z \theta(z,y)\bigwedge_i z_i\in G \wedge \bigwedge_j z_j\notin G.$$
    Let $z_I, z_J$ the tuples of variables consisting of the $z_i$, respectively the $z_j$. Now assume that $(a,b)\models \theta(z,y)\wedge  z_I\subset G \wedge z_J\cap G=\emptyset$ for $a$ from $A'$ and $b$ from $B$. By modularity there exist tuples $a^G$ and $a'$ of elements of $A'$ such that $a^Ga'$ is independent over $B$ in the sense of the pregeometry of $(A',\oplus,\ominus,e)$, such that $\vect{B,a} = \vect{B,a^Ga'}$ and such that $G(\vect{B,a^Ga'}) = \vect{G(B),a^G}$. As $a\subset \vect{B,a^Ga'}$ and by modularity, we can choose finite tuples $c^G,c$ from $B$ such that $c^Gc$ is independent, $a\subset \vect{a^Ga'c^Gc}$ and $G(\vect{a^Ga'c^Gc}) = \vect{a^Gc^G}$. Now as $a_I\subset \vect{a^Gc^G}$ there is an $\set{\oplus,\ominus,e}$-formula $\lambda(a_I,a^G,c^G)$ that witnesses it. As $a_J\subset \vect{a^Ga'c^Gc}\setminus \vect{a^Gc^G}$, for each $j\in J$ there are formulae $\tau_j(z_j,x^G,x',v^G,v)$ strict in $x',v$ (and algebraic in $z_j$) such that for each $j\in J$ we have $\models \tau_j(a_j,a^G,a',c^G,c)$. Let $b' = bc^Gc$ and let $\Lambda(x^Gx',z_J,b')$ be the following formula
    $$\exists z \theta(z,b)\wedge \lambda(z_I,x^G,c^G)\wedge \bigwedge_{j\in J} \tau_j(z_j,x^G,x',c^G,c).$$
    The set of realisations of this $\LL$-formula can be decomposed as a finite union of irreducible quasi-varieties defined over $B$, and the tuple $a^Ga'$ is in one of those, let's assume that the latter is defined by the formula $\varphi(x^Gx',db')$. As the tuple $a^Ga'$ is independent over $B$, we have that $\models \theta_{\varphi}(db')$. By the axioms, there exists a tuple $\tilde a^G \tilde a'$ from $B$ such that $\tilde a^G\subset G(B)$ and all realisations of $\bigwedge_{j\in J} \tau_j(x_j,\tilde a^G,\tilde a',c^G,c)$ are not in $G$. Then the $z$-tuple whose existence is mentionned in $\Lambda$ satisfies the formula $\theta(z,b)\wedge  z_I\subset G \wedge z_J\cap G = \emptyset$. 
  \end{proof}

  The theory $T$ is stable, we denote by $\indi T$ the non-forking independence relation defined over every subset of a monster model of $T$. We denote by $\acl_T(X)$ the algebraic closure of $X\subset A$ in the sense of $T$. Then using \cite[Theorem 4.1, Corollary 4.3]{dE18A} as in the proof of \cite[Theorem 5.28]{dE18A}, we get the following Corollary.
  \begin{cor}
    The theory $TG$ is $\NSOP 1$ and not simple. Furthermore, Kim-independence over models is the relation
    $$X\indi{T}_A Y \text{ and } G(\vect{X,Y}) = \vect{G(X),G(Y)}$$
    for $A\models T$, $X,Y$ $\acl_T$-closed containing $A$.
  \end{cor}
  \begin{proof}
    From \cite[Theorem 4.1]{dE18A} in order to prove that $TG$ is $\NSOP 1$ and to get the description of Kim-independence, it is sufficient to prove that if $X,Y,Z$ are $\acl_T$-closed, containing a model $A$ of $T$, and if $Z\indi T _A X,Y$, then $\vect{\acl_T(XZ),\acl_T(YZ) }\cap \acl_T(XY) = \vect{X,Y}$. Let $X,Y,Z$ be as such, and let $w\in\vect{\acl_T(XZ),\acl_T(YZ) }\cap \acl_T(XY)$. Then there exists $u\in  \acl_T(XZ)$ witnessed by the algebraic formula $\phi(u,m)$ where $m\in Z$ an $\psi(x,z)$ has parameters in $X$ and $v\in \acl_T(YZ)$, witnessed by the algebraic formula $\psi(v,m)$ for $m\in Z$ and $\psi(y,z)$ has parameters in $Y$, such that $[n]w = u\oplus v$. It follows that $\models \exists xy\ [n]w = x\oplus y\wedge \phi(x,m)\wedge \psi(y,m)$. Now by stability, $tp^T(\acl_T(XY)/Z)$ is heir of $tp^T(\acl_T(XY)/A)$ and since $w\in \acl_T(XY)$, there is some $a\in A$ such that $\models \exists xy\ [n]w = x\oplus y\wedge \phi(x,a)\wedge \psi(y,a)$. As $X$ and $Y$ are $\acl_T$-closed, we conclude that $w\in \vect{X,Y}$. The other inclusion being trivial, we conclude that $\vect{\acl_T(XZ),\acl_T(YZ) }\cap \acl_T(XY) = \vect{X,Y}$. Let $a,b,c$ be three generics independent over some model $A$. Then it is clear that $\acl_T(Aac)\supsetneq \vect{\acl_T(Aa),\acl_T(Ac)}$, let $d$ be in $\acl_T(Aac)\setminus\vect{\acl_T(Aa),\acl_T(Ac)}$. Let $d'$ be in $\acl_T(Abc)\setminus \acl_T(Aac)$. Then $d\oplus d'$ is not in $\acl_T(Aac)$ and neither in $\vect{\acl_T(Aa),\acl_T(Abc)}$. We conclude that $\vect{\acl_T(Aac),\acl_T(Abc)}\supsetneq \acl_T(Aac)\cup\vect{\acl_T(Aa),\acl_T(Abc)}$. By \cite[Corollary 4.3]{dE18A}, we conclude that $TG$ is not simple. 
  \end{proof}
  \begin{rk}
    Note that every $\acl_T$-closed set is a model of $T$, so Kim-independence is defined over every $\acl_T$-closed sets. 
  \end{rk}

  \begin{rk}\label{rk_necess}
    A question one might ask is the following: 
    \begin{center}is $\End(A) = \Z$ a necessary condition for $TG$ to exist?\end{center}
    We don't have the answer in the context of abelian varieties, however, if we allow $A$ to be an affine algebraic group we get a negative answer. In ~\cite{dE18A}, $\ACFG$ is the model-companion of an algebraically closed field of fixed positive characteristic with a predicate for an additive subgroup. In the context of this note, it is $TG$ with $T = \ACF_p$ for $p>0$ and $A = \G_a$. There are plenty of nontrivial endomorphisms of $\G_a$ (e.g. $x\mapsto a\cdot x$, $\Frob$,\dots), and this fact does not affect the existence of a model-companion for $T^G$. Assume that $A$ is some abelian algebraic group, if there exists an endomorphisms $\theta\in \End(A)\setminus \Z$, then it is an easy exercise to show that in any existentially closed model of $T^G$, the image $\theta (G)$ is not included in $G$, in fact, $\theta(G)$ is generic in $A$. In particular, if $\End(A)$ and its action on $A$ are definable, then the stabiliser of $G$ under the action of $\End(A)$ is definable. In particular, if $A=\G_a^n$ and the ambient field is of characteristic $0$, then the stabiliser of $G$ is $\Z$ (or $\Q$ if $G$ is assumed divisible) and the model-companion of $T^G$ does not exist (see~\cite[Subsection 5.6]{dE18A} for the case $n = 1$). Conversely, if the endomorphism ring is definable and infinite, then it is an algebraically closed field, hence the group $A$ is a $K$-vector space, hence isomorphic to $\G_a^n$. 
  \end{rk}


  \section{A more general case: generic $\End(A)$-submodule of $A$.}
  
  We end this note by a discussion on a slightly more general setting. We still assume that $A$ is simple. Let $R$ be a ring of definable endomorphisms of $A$, and extend the language $\LL$ to $\LL_* = \LL\cup \set{\lambda_r\mid r\in R}$ by function symbols $\lambda_r$ for each element of $R$, and let $T_*$ be the natural expansion by definition to the language $\LL_*$ of the $\LL$-theory $T$. Let $T^A_*$ be the theory of $(A, \oplus,\ominus,e,(\lambda_r)_{r\in R})$, it is a reduct of $T^*$. Expand also $\LL^G_* = \LL^G\cup \set{\lambda_r\mid r\in R}$ in the same way and let $T^G_*$ be the $\LL^G_*$-theory $T^G$ in which each $\lambda_r$ is interpreted as the corresponding endomorphism of $A$, and $(G,\oplus,\ominus,e,(\lambda_r)_{r\in R})$ is an $R$-submodule of $A$, model of $T^A_*$. In order to get that $T^G_*$ has a model-companion using~\cite[Theorem 1.5, Remark 1.8]{dE18A}, it is sufficient to have that 
  \begin{enumerate}
    \item $T_*$ has quantifier elimination in $\LL_*$, which is clear;
    \item $T^A_*$ has quantifier elimination in $\set{\oplus,\ominus,e, (\lambda_r)_{r\in R}}$, is strongly minimal, and the algebraic closure defines a modular pregeometry;
    \item the analogue of Proposition~\ref{prop_H4}, replacing “independent in the sense of the pregeometry in $(A,\oplus,\ominus,e)$” by “independent in the sense of the $R$-module $A$”, which easily follows from Proposition~\ref{prop_free} and Fact~\ref{fact_free}, as in the proof of Proposition~\ref{prop_H4}.
\end{enumerate}
Only $(2)$ needs to be checked. It should follow from the fact that in a simple variety, every endomorphism is onto and has finite kernel. Then one should conclude that $T^G_*$ has a model-companion.\\

\noindent \textbf{Acknowledgement.} The author is grateful to Alf Onshuus and Martin Bays for fruitful discussions on that subject, and to Magnus Carlson for his help. This note is the result of fruitful discussions during a visit in the Universidad de Los Andes in Bogotà. The author is grateful to the ECOS Nord program for supporting this visit.

 \bibliographystyle{alpha}
\bibliography{biblio}

\end{document}